\documentclass[11pt,a4paper]{amsart}
\usepackage{amssymb}
\usepackage{iftex}

\iftutex
  
  \usepackage{unicode-math}
\else

  \usepackage[T1]{fontenc}

\fi

\usepackage{mathrsfs}

\usepackage[colorlinks=true,linkcolor=blue, citecolor=blue, urlcolor=blue,%
pagebackref=true]{hyperref}
\makeatletter

\@addtoreset{equation}{section}
\def\theequation{\thesection.\@arabic \c@equation}

\def\@citecolor{blue}
\def\@linkcolor{blue}
\def\@urlcolor{blue}
\def\theenumi{\@alph\c@enumi}

\makeatother
\theoremstyle{plain}
\newtheorem{theorem}[equation]{Theorem}
\newtheorem{lemma}[equation]{Lemma}
\newtheorem{corollary}[equation]{Corollary}
\newtheorem{proposition}[equation]{Proposition}

\theoremstyle{definition}

\newtheorem{remark}[equation]{Remark}
\newenvironment{remarkbox}[1][]{%
    \begin{remark}[#1] \pushQED{\qed}}{\popQED \end{remark}}

\newtheorem{example}[equation]{Example}

\newtheorem{definition}[equation]{Definition}

\newtheorem{notation}[equation]{Notation}

\newtheorem{discussion}[equation]{Discussion}

\newtheorem{observation}[equation]{Observation}
\newenvironment{observationbox}[1][]{%
    \begin{observation}[#1]\pushQED{\qed}}{\popQED \end{observation}}

\newtheorem{construction}[equation]{Construction}

    {\setcounter{step}{0}}{} \newcounter{step}

\newcommand{\bfa}{\mathbf a}

\newcommand{\bfb}{\mathbf b}

\newcommand{\calD}{\mathcal D}

\newcommand{\calJ}{\mathcal J}

 \let\strSh\calO

\newcommand{\bfx}{\mathbf x}

\def\to{\longrightarrow}
\DeclareMathOperator{\rank}{rk}

\DeclareMathOperator{\coker}{coker}

\DeclareMathOperator{\projective}{\mathbb{P}}
\DeclareMathOperator{\height}{ht}

\DeclareMathOperator{\charact}{char}

\DeclareMathOperator{\Spec}{Spec}
\DeclareMathOperator{\Proj}{Proj}

\DeclareMathOperator{\homology}{H}
\newcommand{\define}[1]{\emph{#1}}
\newcommand{\minus}{\ensuremath{\smallsetminus}}
\makeatletter
\def\RDerChar{\mathbf{R}}
\def\RDer{\@ifnextchar[{\R@Der}{\ensuremath{\RDerChar}}}
\def\R@Der[#1]{\ensuremath{\RDerChar^{#1}}}

\makeatother

\newif\ifreadkumminibib
\readkumminibibfalse

\begin{document}
\title[Maximal minors of $1$-generic matrices]{Maximal minors of $1$-generic matrices have rational singularities}
\author[T.~Chau]{Trung Chau}
\address{Chennai Mathematical Institute, Siruseri, Tamilnadu 603103. India}
\email{chauchitrung1996@gmail.com}

\author[M.~Kummini]{Manoj Kummini}
\address{Chennai Mathematical Institute, Siruseri, Tamilnadu 603103. India}
\email{mkummini@cmi.ac.in}


\thanks{Both authors were partially supported by an Infosys Foundation
fellowship.}
\begin{abstract}
We show that the quotient ring by the ideal of maximal minors of a
$1$-generic matrix has rational singularities.
This answers a conjecture of Eisenbud (1988) that such rings are normal,
and generalizes a result of
Conca, Mostafazadehfard, Singh and Varbaro (2018)
that generic Hankel determinantal rings have 
rational singularities in characteristic zero.
\end{abstract}

\maketitle

\section{Introduction}
Let $\Bbbk$ be a field and $M$ a matrix whose entries are $\Bbbk$-linear
combinations of a finite set of indeterminates over $\Bbbk$.
For example, $M$ could be a generic matrix, a symmetric matrix of
variables or a Hankel matrix of variables.
(Definitions will be given in Section~\ref{section:prelim}.)
Write $\Bbbk[M]$ for the standard-graded
polynomial ring over $\Bbbk$ generated by the entries of $M$.
By a \emph{determinantal ring}, we mean a ring of the form
$\Bbbk[M]/I_t(M)$, where $t$ is positive 
integer and $I_t(M)$ denotes the ideal generated by the $t \times t$ minors
of $M$.
Determinantal rings arise in many natural geometric questions
and form an important class of commutative rings.
We refer to \cite{BrVeDetRings}, \cite{WeymSyzygies03} and
\cite{BCRVDetRings} for in-depth treatment of these algebras.

Many determinantal rings have good singularities. For example, 
if $M$ is generic and $\charact \Bbbk > 0$, then $\Bbbk[M]/I_t(M)$ is
$F$-regular~\cite[(7.14)]{HochsterHunekeTCParamIdealsSplitting1994}.
If $M$ is generic and $\charact \Bbbk = 0$, then $\Bbbk[M]/I_t(M)$ can be
realized as the invariant ring of a reductive group, and, therefore, has
rational singularities~\cite{BoutotSingRatl1987}. (This follows also from
the geometric method of Kempf-Lascoux-Weyman of computing syzygies;
see~\cite[6.1.5]{WeymSyzygies03}.)
If $M$ is generic symmetric and $\charact \Bbbk = 0$, then 
$\Bbbk[M]/I_t(M)$ has rational singularities~\cite[\S 6.3]{WeymSyzygies03}.
Recently 
Conca, Mostafazadehfard, Singh and Varbaro~\cite{CMSVHankel}
showed that generic Hankel determinantal rings have $F$-rational
singularities (if $\charact \Bbbk$ is sufficiently large)
or rational singularities (in characteristic zero).

In this paper, we consider $1$-generic matrices~\cite{EisHankel}.
Let $M$ be an $m \times n$ matrix whose entries, as earlier, are
$\Bbbk$-linear combinations of indeterminates. We say that $M$ is
\define{$1$-generic} if
for every non-zero $(\lambda_1, \ldots, \lambda_m) \in \Bbbk^m$ and every
non-zero $(\mu_1, \ldots, \mu_n) \in \Bbbk^n$
\[
\begin{bmatrix}
\lambda_1 & \cdots & \lambda_m
\end{bmatrix}
M
\begin{bmatrix}
\mu_1 \\ \vdots \\ \mu_n
\end{bmatrix}
\neq 0
\]
Generic, generic symmetric and generic Hankel matrices
are examples of $1$-generic matrices. They correspond to certain linear
sections of a generic determinantal variety. In geometry, they arise while
studying the product map on global sections of two line bundles.
This property was used in~\cite{EisKohStillmanDeterEqnsCurves1988}
and~\cite{SS11} to show certain varieties are defined by the $2 \times 2$
minors of $1$-generic matrices.
At the same time, in general, the ideals defined by submaximal minors of a
$1$-generic matrix need not have good singularities: they need not be prime,
could have embedded components, etc.~\cite{GuerrSwansonMinors2003}.

In this paper, we prove the following theorem:
\begin{theorem}
\label{theorem:main}
Let $\Bbbk$ be a field.
Let $M$ be an $m \times n$ $1$-generic matrix, with $m \leq n$.
Then $\Bbbk[M]/I_m(M)$ has rational singularities.
\end{theorem}

This answers partially in the affirmative a conjecture of
Eisenbud~\cite[top of p.~553]{EisHankel} that 
$\Bbbk[M]/I_m(M)$ is normal.
Further the above theorem generalizes the part of
\cite[Theorem~2.1]{CMSVHankel} that deals with
rational singularities. Our proof is an elementary
application of the Kempf-Lascoux-Weyman geometric
method~\cite[Chapter~5]{WeymSyzygies03}.

\subsection*{Acknowledgements}
The computer algebra system~\cite{M2} 
provided valuable assistance in studying examples.

\section{Preliminaries}
\label{section:prelim}

Throughout this paper $\Bbbk$ denotes a field. 
We say that an $m\times n$ matrix (with $m \leq n$) is 
\define{generic} if the $mn$ entries
are algebraically independent over $\Bbbk$. A square symmetric matrix 
in which the entries on and above the diagonal are algebraically
independent over $\Bbbk$ is called a \define{generic symmetric} matrix.
A \define{generic Hankel} matrix is one of the form
\[
M = 
\begin{bmatrix}
x_1 & x_2 & \cdots & x_n \\
x_2 & x_3 & \cdots & x_{n+1} \\
\vdots & \ddots \\
x_m & x_{m+1} & \cdots & x_{m+n-1} \\
\end{bmatrix}
\]
where the $x_i$ are algebraically independent over $\Bbbk$. It is immediate
from the definition of $1$-generic matrices that generic, generic symmetric
and generic Hankel matrices are $1$-generic.

\subsection*{Rational singularities}

By a 
\define{variety} over $\Bbbk$ we mean a quasi-projective irreducible and
reduced $\Bbbk$-scheme.

\begin{definition}
\label{definition:ratl}
(\cite[2.76]{KollarSingMMP13})
Let $X$ be a variety. A \define{rational resolution} of $X$ is
a proper birational morphism $f : Z \to X$ such that 
\begin{enumerate}

\item
$Z$ is non-singular;

\item
$f_* \strSh_{Z} = \strSh_X$;

\item
$\RDer^i f_* \strSh_{Z} = 0$ for all $i>0$;

\item
$\RDer^i f_* \omega_{Z} = 0$ for all $i>0$.
\end{enumerate}
We say that $X$ has \define{rational singularities} if $X$ has 
a rational resolution.
\end{definition}

If $X$ has a rational resolution, then every resolution of
singularities of $X$ is a rational resolution; see the remark
after~\cite[2.76]{KollarSingMMP13}.

\begin{remarkbox}
\label{remarkbox:OXenough}
Since $X$ is quasi-projective, it admits a dualizing complex
$\calD^\bullet$. Then, by~\cite[(4.2)]{LipCMgraded94}, 
conditions (b) and (c) of the above definition together is equivalent
to the statement that $\RDer f_* f^!\calD^\bullet \to \calD^\bullet$ is a
quasi-iso\-mor\-phism.
Now suppose that $X$ is Cohen-Macaulay, or, equivalently, that
$\calD^\bullet$ has exactly one non-zero cohomology, at the left end.
Hence, in this case, it is not necessary to check
the condition (d) of the above definition.
\end{remarkbox}

We will use the following proposition to see that if $X$ has rational
singularities, then so does a geometric line bundle over it.

\begin{proposition}
\label{proposition:smOverRatl}
Let $X$ be a Cohen-Macaulay variety with a rational resolution 
$f : Z \to X$.
Let $g: X' \to X$ be a smooth morphism of varieties.
Then the projection morphism $f' : X' \times_X Z \to X'$ is a rational
resolution.
\end{proposition}

\begin{proof}
Write $Z' = X' \times_X Z$.
Note that $f'$ is proper~\cite[II.4.8]{HartAG} and that,
since $g$ is dominant, $f'$ is birational.
Since $Z$ is non-singular and the projection morphism $Z' \to Z$ is
smooth, $Z'$ is non-singular~\cite[III.10.1]{HartAG}.
Conditions (b) and (c) of Definition~\ref{definition:ratl} (for the map
$f'$) follow from flat base-change~\cite[III.9.3]{HartAG}.
Since $g$ is smooth, $X'$ is Cohen-Macaulay.
Now use Remark~\ref{remarkbox:OXenough}.
\end{proof}

\section{Proof of the theorem}

Let $x_0, \ldots, x_r$ be indeterminates and $S = \Bbbk[x_0, \ldots, x_r]$.
Let $M$ be an $m \times n$ (with $m \leq n$) $1$-generic matrix 
of linear forms $\ell_{i,j}$ in $S$.
Write $R = S/I_m(M)$. 
For positive integers $t$, write $\projective^t = \projective^t_\Bbbk$.
Let $X = \Proj R \subseteq \projective^r = \Proj S$.

\begin{remarkbox}
\label{remarkbox:fromEisSyz}
We collect a few facts from~\cite[Theorem~6.4]{EisSyz05}. 
$I_m(M)$ is a prime ideal of height $(n-m+1)$.
$R$ is a Cohen-Macaulay and its minimal free resolution
over $S$ is given by the Eagon-Northcott complex. In particular, the
Castelnuovo-Mumford regularity of $R$ is $m-1$.
\end{remarkbox}

\begin{observationbox}
\label{observationbox:dimS}
With notation as above, $r+1 \geq m+n-1$;
equality is attained if $M$ is a Hankel matrix.
To see this, let $\lambda_i, 1 \leq i \leq m$ and $\mu_j, 1 \leq j \leq n$
be in $\Bbbk$ such that
\[
\begin{bmatrix}
\lambda_1 & \cdots & \lambda_m
\end{bmatrix}
M
\begin{bmatrix}
\mu_1 \\ \vdots \\ \mu_n
\end{bmatrix}
=0
\]
This gives $r+1$ bilinear equations in the $\lambda_i$ and the $\mu_j$.
The solution set inside $\projective^{m-1} \times \projective^{n-1}$ has
codimension at most $r+1$. Since $M$ is $1$-generic, the solution set 
inside $\projective^{m-1} \times \projective^{n-1}$
is empty. Hence $r+1 > m+n-2$.
Combining this with the information on $\height I_m(M)$
(Remark~\ref{remarkbox:fromEisSyz}) we see that
\begin{equation}
\label{equation:dimR}
\dim R = (r+1)-(n-m+1) \geq 2m-2.
\qedhere
\end{equation}
\end{observationbox}

Write $\projective = \projective^{r} \times \projective^{m-1}$.
Let $y_1, \ldots, y_m$ be homogeneous coordinates for $\projective^{m-1}$.
Let $Z \subseteq \projective$ be the closed subscheme defined by the $n$
bilinear forms $\sum_{i=1}^m y_i l_{i,j}, 1 \leq j \leq n$.
Let $\pi_1, \pi_2$ denote the projection maps $\projective \to
\projective^{r} $ and $\projective \to \projective^{m-1} $ respectively.

\begin{lemma}
\label{lemma:ZKoszul}
The Koszul complex
\[
0 \rightarrow \strSh_{\projective}(-n,-n) \rightarrow
\strSh_{\projective}(-n+1,-n+1)^n \rightarrow \cdots \rightarrow
\strSh_{\projective}(-1,-1)^n \rightarrow 
\strSh_{\projective} \rightarrow 0
\]
corresponding to the $n$ bilinear forms above is a locally free resolution
of $\strSh_Z$ on $\projective$.
\end{lemma}

\begin{proof}
It suffices to show that at all closed points $p \in Z$,
the $n$ bilinear forms define an ideal of height $n$ in
$\strSh_{\projective,p}$.
Since extending the base field is faithfully flat, we may assume that
$\Bbbk$ is algebraically closed. 

We will prove a slightly stronger statement: 
Let $\bfb = [b_1 : \cdots :b_m] \in \projective^{m-1}$.
Then there exists 
a neighbourhood $U$ of $\bfb$ such that $\pi_2^{-1}U$ has the structure of
a projective bundle over $U$ with each fibre isomorphic to 
$\projective^{r-n}$, sitting as
a linear subvariety inside the ambient fibre $\projective^{r}$.

Write $M = \sum_{k=0}^{r} M_k x_k$, where the $M_k$ are $(m \times n)$
matrices with entries in $\Bbbk$. 
Let $A_\bfb$ be the $n \times (r+1)$
matrix whose $k$th column is 
\[
\left(\begin{bmatrix} b_1 & \cdots & b_m \end{bmatrix}M_k\right)^{%
\mathrm{tr}}.
\]
Write $\projective^r(\bfb)$ for the fibre over $\bfb$.
Then $Z \cap \projective^r(\bfb)$ is defined inside $\projective^r(\bfb)$
by the equation
$\displaystyle A_\bfb \begin{bmatrix} x_0 \\ \vdots \\ x_r
\end{bmatrix} = 0$.
Since $M$ is $1$-generic, $\rank A_\bfb = n$.
Let $0 \leq k_1 < \cdots < k_n \leq r$ be such that the columns
\[
\left(\begin{bmatrix} b_1 & \cdots & b_m \end{bmatrix}M_{k_l}\right)^{%
\mathrm{tr}}, 1 \leq l \leq n
\]
of $A_\bfb$ are linearly independent, i.e., 
the determinant of the $n \times n$ submatrix
of $A_\bfb$ given by these columns (and all the rows) is non-zero.
This determinant is a polynomial function of $(b_1, \ldots, b_m)$. 
Hence there is a neighbourhood $U$ of $\bfb$ such that for all $\bfb' :=
[b'_1:\cdots :b'_m] \in U$, the columns of $A_{\bfb'}$ given by 
\[
\left(\begin{bmatrix} b'_1 & \cdots & b'_m \end{bmatrix}M_{k_l}\right)^{%
\mathrm{tr}}, 1 \leq l \leq n
\]
are linearly independent. 
Therefore for all $\bfb'  \in U$, the solutions to 
$\displaystyle A_{\bfb'} \begin{bmatrix} x_0 \\ \vdots \\ x_r
\end{bmatrix} = 0$
are given by taking
arbitrary values for $a_k, k \in \{0, \ldots, r \} \minus \{k_1, \ldots,
k_n\}$ (which, in turn, determine the values of $a_k, k \in \{k_1, \ldots,
k_n\}$ uniquely). This proves the assertion made above about 
$\pi_1^{-1}U$ and completes the proof of the lemma.
\end{proof}

Suppose that $M$ is any $m \times n$ matrix, not necessarily $1$-generic, of
linear forms in $S$. Consider the Koszul complex as in the above lemma. The
proof shows that 
it is a resolution
if and only if
the matrix $A_\bfb$ has rank $n$ for every $\bfb$, which holds 
if and only if 
$M$ is $1$-generic.

\begin{proposition}
\label{proposition:OZdirImage}
${\pi_1}_*\strSh_{Z}$ has a locally free resolution
\[
0 \rightarrow \strSh_{\projective^{m-1}}(-n)^{b_{n-m+1}} \rightarrow
\strSh_{\projective^{m-1}}(-n+1)^{b_{n-m} } \rightarrow \cdots \rightarrow
\strSh_{\projective^{m-1}}(-m)^{b_1} \rightarrow 
\strSh_{\projective^{m-1}} \rightarrow  0
\]
on $\projective^{m-1}$ with 
\[
b_i = \binom{n}{i+m-1} \rank_\Bbbk 
\homology^{m-1}(\projective^{m-1}, \strSh_{\projective^{m-1}}(1-i-m)).
\]
Further, $\RDer^j{\pi_1}_*\strSh_{Z} = 0$ for all $j>0$.
\end{proposition}

\begin{proof}
Write $K_\bullet$ for the Koszul complex of Lemma~\ref{lemma:ZKoszul}
placed on the non-positive horizontal axis, i.e., 
$K_{i} := \strSh_{\projective}(-i,-i)^{\binom ni}$ 
at position $(-i,0)$, $0 \leq i \leq n$. 
We look at the hypercohomology spectral sequence (in the second quadrant)
for computing
$\RDer {\pi_1}_*K_\bullet$.
Then we have a spectral sequence
\[
E_1^{-i,j} = \RDer^j {\pi_1}_* K_{i} \Rightarrow 
\homology^{j-i}(\RDer{\pi_1}_* K_\bullet)
\simeq
\RDer^{j-i}{\pi_1}_* \strSh_Z,
\]
where the last isomorphism follows from the quasi-isomorphism $K_\bullet
\to \strSh_Z$.
Note that
\[
\RDer^j {\pi_1}_* K_{i} = 
\begin{cases}
\strSh_{\projective^{m-1}}, & \text{if}\; (i,j ) = (0,0 );
\\
\strSh_{\projective^{m-1}}(-i)^{b_{i-m+1}}, & \text{if}\; m \leq i \leq n
\;\text{and}\; j=m-1;
\\
0, & \text{otherwise}.
\end{cases}
\]
Therefore $E_1^{-i,j}  = 0$ unless $j-i \leq 0$, from which it follows
that 
$\RDer^{j}{\pi_1}_* \strSh_Z = 0$ for all $j > 0$.

Again from the quasi-isomorphism 
$\RDer{\pi_1}_* K_\bullet \to \RDer{\pi_1}_* \strSh_Z$,
it follows that 
$\homology^{j-i}(\RDer{\pi_1}_* K_\bullet) = 0$ if $j<i$.
Hence
$E_\infty^{-i,j} = 0$ for all $j<i$.
For $i>m$, 
$E_\infty^{-i,m-1} = E_2^{-i,m-1}$, so, 
$E_1$-page maps
\[
0 \rightarrow E_1^{-n,m-1} 
\rightarrow E_1^{-n+1,m-1} \rightarrow \cdots
\rightarrow E_1^{-m,m-1}
\]
form an exact sequence, except at the right end.
On the other hand, 
\[
E_m^{-m,m-1} = E_2^{-m,m-1}
= \coker \left(
E_1^{-m-1,m-1}
\rightarrow E_1^{-m,m-1}
\right )
\]
Since 
$E_\infty^{-m,m-1} = E_{m+1}^{-m,m-1}$ and
$E_\infty^{0,0} = E_{m+1}^{0,0}$,
we further see that 
the $E_{m}$-page map
$E_{m}^{-m,m-1} \rightarrow E_{m}^{0,0} = E_{1}^{0,0}$
is injective and its cokernel is 
${\pi_1}_* \strSh_Z$.
Putting these together, we conclude that
\[
0 \rightarrow E_1^{-n,m-1} 
\rightarrow E_1^{-n+1,m-1} \rightarrow \cdots
\rightarrow E_1^{-m,m-1}
\rightarrow E_1^{0,0}
\to 0
\]
is a locally free resolution of 
${\pi_1}_* \strSh_Z$.
\end{proof}

\begin{corollary}
\label{corollary:schThyImage}
$\strSh_X  = {\pi_1}_*\strSh_{Z}$.
\end{corollary}

\begin{proof}
Proposition~\ref{proposition:OZdirImage} implies that there exists a sheaf
of ideals $\calJ \subseteq \strSh_{\projective^r }$ 
such that $\strSh_{\projective^r }/\calJ = {\pi_1}_*\strSh_{Z}$.
Since $Z$ is an integral scheme, $\calJ$ is a sheaf of prime ideals.
On the other hand, $\pi_1$ is proper, so  $\pi_1(Z) = X$. 
Therefore the support of ${\pi_1}_*\strSh_{Z}$ is $X$, which implies that
$\calJ$ is the ideal sheaf of the irreducible and reduced subscheme $X$
(note that $R$ is a domain).
Hence $(X, {\pi_1}_*\strSh_{Z}) = (X, \strSh_X)$.
\end{proof}

Eisenbud~\cite[top of p.~553]{EisHankel} conjectured that $1$-generic
subvarieties (to be precise, something a little more general than
$1$-generic) are normal. The next proposition answers this in the
affirmative.
\begin{proposition}
\label{proposition:normal}
$R$ is a normal domain.
\end{proposition}

\begin{proof}
Use Corollary~\ref{corollary:schThyImage} and the fact that $Z$ is normal
to see that $X$ is normal.
Therefore $R$ is regular in codimension $1$. Since it is
Cohen-Macaulay (Remark~\ref{remarkbox:fromEisSyz}), it is a normal domain.
\end{proof}

\begin{proposition}
\label{proposition:desing}
The map $Z \to X$ is a desingularization.
\end{proposition}

\begin{proof}
Note that $Z$ is non-singular.
We need to show that the map is birational.
We show that it is generically finite and use the next lemma.
Let $M'$ be the $(m-1) \times n$ submatrix of $M$ consisting of the first
$m-1$ rows. It is $1$-generic, so $\height I_{m-1}(M') = n-m+2$.
Therefore $\height I_{m-1}(M) \geq n-m+2$.
Since $\height I_m(M) = n-m+1$, there exists a dense open subset $U$ of $X$
such that for all $\bfa := [a_0 : \cdots : a_r] \in U$ the specialization
$M|_{\bfx = \bfa}$ has rank $m-1$. 
For all such $\bfa$, there exists a unique $\bfb \in \projective^{m-1}$
such that $(\bfa,\bfb) \in Z$.
\end{proof}

\begin{lemma}
Let $f : Y \to Y'$ be a dominant finite-type generically finite map of
irreducible and reduced noetherian schemes such that $f_* \strSh_Y =
\strSh_{Y'}$.
Then $f$ is birational.
\end{lemma}

\begin{proof}
Using~\cite[Exercise~II.3.8]{HartAG}, let $U' \subseteq Y'$ be an affine
open subscheme such that, with $U = f^{-1}(U')$, $U \to U'$ is finite.
Hence $U$ is affine, and the map $U \to U'$ is given by a finite map of
rings $A' \to A$.
Now, by the hypothesis, $A' = \Gamma(U',\strSh_{Y'}) = \Gamma(U,\strSh_Y) =
A$.
Therefore $f$ is birational.
\end{proof}

\begin{proof}[Proof of Theorem~\protect{\ref{theorem:main}}]
We have already seen that $R$ is normal
(Proposition~\ref{proposition:normal}), so we now argue that it has
rational singularities. We first show that $X$ has rational singularities.
We see that the morphism $\pi_1 : Z \to X$ satisfies the first three
conditions of Definition~\ref{definition:ratl}, using
Proposition~\ref{proposition:OZdirImage},
Corollary~\ref{corollary:schThyImage}
and Proposition~\ref{proposition:desing}.
Since $X$ is Cohen-Macaulay (and admits a dualizing module), we 
use~\cite[(4.2)]{LipCMgraded94} to see that the final condition also holds.
Hence $X$ has rational singularities.

The $a$-invariant of $R$ is negative, since  
$\dim R  \geq 2m-2$ and 
the Castelnuovo-Mumford regularity of $R$ is $m-1$
(see~\eqref{equation:dimR} and Remark~\ref{remarkbox:fromEisSyz}).
Now the theorem now essentially follows
from~\cite[Theorem~1.5]{HyryBlowupRingsRationalSings1999}, but we sketch
some details since that theorem and
the preceding definitions there assume characteristic zero.

Let $A$ be the Rees algebra $R \oplus R_+t \oplus (R_+)^2t^2 \oplus
\cdots$ and $Y = \Proj A$.
Note that $Y$ is a line bundle over $X$~\cite[II, (8.7.8)]{EGA}, so $Y$ has
a rational resolution $Z' \to Y$ (Proposition~\ref{proposition:smOverRatl}). 
Use the Leray spectral sequence for the composite map $Z' \to Y \to \Spec R$ 
(as the proof of~\cite[Theorem~1.5]{HyryBlowupRingsRationalSings1999} does)
to conclude that the conditions
(b) and (c) of Definition~\ref{definition:ratl} are satisfied for the
composite map $Z' \to \Spec R$.
Using fact that $R$ is Cohen-Macaulay (and Remark~\ref{remarkbox:OXenough})
we conclude that $R$ has rational singularities.
\end{proof}

\ifreadkumminibib
\bibliographystyle{alphabbr}
\bibliography{kummini}
\else

\fi %

\end{document}